\documentclass[reqno,11pt]{amsart}
\usepackage[letterpaper, margin=1in]{geometry}
\RequirePackage{amsmath,amssymb,amsthm,graphicx,mathrsfs,url,slashed,subcaption}
\RequirePackage[usenames,dvipsnames]{xcolor}
\RequirePackage[colorlinks=true,linkcolor=Red,citecolor=Green]{hyperref}
\RequirePackage{amsxtra}
\usepackage{cancel, longtable}
\usepackage{tikz, wrapfig}

\makeatletter
\@namedef{subjclassname@2020}{\textup{2020} Mathematics Subject Classification}
\makeatother

\title{Lipschitz maps with prescribed local Lipschitz constants}
\author{Aidan Backus}
\address{Brown University, Providence, RI}
\email{aidan\_backus@brown.edu}
\author{Ng Ze-An}
\address{University of Malaya, Kuala Lumpur}
\email{ngzean2@hotmail.com}
\date{\today}

\newcommand{\NN}{\mathbf{N}}

\newcommand{\RR}{\mathbf{R}}

\newcommand*\dif{\mathop{}\!\mathrm{d}}

\DeclareMathOperator{\card}{card}
\DeclareMathOperator{\dist}{dist}

\DeclareMathOperator{\Hom}{Hom}

\DeclareMathOperator{\supp}{supp}

\newcommand{\vol}{\mathrm{vol}}

\newcommand{\Lip}{\mathrm{Lip}}
\newcommand{\Riem}{\mathrm{Riem}}

\DeclareMathOperator{\Fin}{Fin}

\newcommand{\dfn}[1]{\emph{#1}\index{#1}}

\newtheorem{theorem}{Theorem}[section]

\newtheorem{lemma}[theorem]{Lemma}

\newtheorem{proposition}[theorem]{Proposition}
\newtheorem{corollary}[theorem]{Corollary}

\newtheorem{assumption}[theorem]{Assumption}

\newtheorem*{claim}{Claim}

\theoremstyle{definition}
\newtheorem{definition}[theorem]{Definition}

\newtheorem{problem}[theorem]{Problem}



\numberwithin{equation}{section}

\def\Xint#1{\mathchoice
{\XXint\displaystyle\textstyle{#1}}%
{\XXint\textstyle\scriptstyle{#1}}%
{\XXint\scriptstyle\scriptscriptstyle{#1}}%
{\XXint\scriptscriptstyle\scriptscriptstyle{#1}}%
\!\int}
\def\XXint#1#2#3{{\setbox0=\hbox{$#1{#2#3}{\int}$ }
\vcenter{\hbox{$#2#3$ }}\kern-.6\wd0}}

\def\dashint{\Xint-}

\usepackage[backend=bibtex,style=alphabetic,giveninits=true]{biblatex}

\renewbibmacro{in:}{}
\DeclareFieldFormat{pages}{#1}

\begin{document}
\begin{abstract}
Let $\Gamma$ be a closed subset of a complete Riemannian manifold $M$ of dimension $\geq 2$, let $f: M \to N$ be a Lipschitz map to a complete Riemannian manifold $N$, and let $\psi$ be a continuous function which dominates the local Lipschitz constant of $f$.
We construct a Lipschitz map which agress with $f$ on $\Gamma$ and whose local Lipschitz constant is $\psi$.
\end{abstract}

\maketitle

\section{Introduction}
A basic problem in metric geometry concerns the extension of Lipschitz maps: given metric spaces $X, Y$, a closed set $\Gamma \subset X$ and a Lipschitz map $f: \Gamma \to Y$, can we extend $f$ to a Lipschitz map $u: X \to Y$ so that ``$u$ does not stretch $X$ more than $f$ stretches $\Gamma$"?
This problem is partially resolved by the \dfn{Kirszbraun--Valentine theorem} \cite{Lang1997,Kassel17}, which, under a suitable hypothesis on the Alexandrov curvatures of $X, Y$, furnishes an extension $u$ whose global Lipschitz constant
$$\Lip(u, X) := \sup_{\substack{x, y \in X \\ x \neq y}} \frac{\dist_Y(u(x), u(y))}{\dist_X(x, y)}$$
is as small as possible.
For example, if $\Gamma$ is a closed subset of $\RR^d$ and $f: \Gamma \to \RR^D$ is a Lipschitz map, then $f$ admits an extension $u$ to all of $\RR^d$ such that 
$$\Lip(u, \RR^d) = \Lip(f, \Gamma).$$
We refer the reader to \cite[\S3]{Kassel17} for an elegant proof of the Kirszbraun-Valentine theorem in the setting of hyperbolic geometry; the same argument also applies in the euclidean setting.

Because the Kirszbraun--Valentine extension $u$ of $f$ can be nonunique, it is natural to impose further conditions on $u$ by constraining the \dfn{local Lipschitz constants}
$$Lu(x) := \lim_{r \to 0} \Lip(u, B(x, r)).$$
Previous work has considered maps which minimize $Lu$ in various senses; such maps are called \dfn{tight} or \dfn{$\infty$-harmonic} \cite{Crandall2008,Sheffield12,Naor2012}.
The $\infty$-harmonic scalar fields are exactly the viscosity solutions of the \dfn{$\infty$-Laplacian}
$$\langle \nabla^2 u, \nabla u \otimes \nabla u\rangle = 0.$$
The existence and uniqueness of tight maps is a fascinating and challenging open question \cite[\S5]{Sheffield12}.

In this paper, we complement the study of tight maps by exhibiting a sort of extreme nonuniqueness for the Kirszbraun--Valentine theorem.
Given any continuous function $\psi$ which is greater than the local Lipschitz constant of a Lipschitz extension of $f$, we construct a Lipschitz extension of $f$ with local Lipschitz constant $\psi$.
More precisely, we show:

\begin{theorem}\label{eikonal extension}
Let $M$ be a complete Riemannian manifold such that $\dim M \geq 2$, and let $N$ be a complete Riemannian manifold.
Let $\Gamma \subseteq M$ be a closed set and $\psi: M \setminus \Gamma \to \RR_+$ be a continuous function.
Then for every Lipschitz map $f: M \to N$ such that $Lf \leq \psi$ on $M \setminus \Gamma$, there exists a Lipschitz map $u: M \to N$, homotopic to $f$ relative to $\Gamma$, such that
\begin{equation}\label{eikonal equation}\
Lu = \psi \text{ on } M \setminus \Gamma.
\end{equation}
\end{theorem}

By taking $\Gamma = \emptyset$, $N = \RR$, and $f = 0$ in Theorem \ref{eikonal extension}, we conclude:

\begin{corollary}\label{continuous functions are local Lips}
For every continuous function $\psi: M \to \RR_+$, there exists a Lipschitz function $u: M \to \RR$ with $Lu = \psi$.
\end{corollary}

Corollary \ref{continuous functions are local Lips} does not characterize those functions which can appear as local Lipschitz constants.
Indeed, $u(x, y) := 1_{x > 0} x$
has $Lu(x, y) = 1_{x \geq 0}$.
If $\psi$ appears as a local Lipschitz constant, then $\psi$ is necessarily upper-semicontinuous and ``lower-semicontinuous on average" (see Lemma \ref{lower semicontinuity on average}).
However, we do not know if this condition is sufficient; our method of proof does not seem to be able to show that such a function appears as a local Lipschitz constant.
We therefore pose:

\begin{problem}
Let $\psi: M \to \RR_+$ be an upper-semicontinuous function.
What conditions on $\psi$ are necessary and sufficient for there to exist a Lipschitz function $u: M \to \RR$ such that $Lu = \psi$?
\end{problem}

Let $u: M \to N$ be a map between closed manifolds which minimizes its Lipschitz constant in its homotopy class.
In applications to Teichm\"uller theory and its generalizations by the Thurston school, the \dfn{stretch set} $\{Lu = \Lip(u, M)\}$ of the map $u$ plays a crucial role \cite{Kassel17, daskalopoulos2022transverse,thurston1998minimal}.
Theorem \ref{eikonal extension} shows that any closed set which contains the so-called ``maximally-stretched geodesic lamination" (see \cite[Theorem 5.1]{Kassel17}) can appear as a stretch set.

\subsection{Idea of the proof}
It is tempting to look for solutions of the eikonal equation (\ref{eikonal equation}) using Perron's method.
Of course, if $\dim N \geq 2$, then (\ref{eikonal equation}) is a strongly coupled system, so the usual formulation of Perron's method \cite{Ishii92} fails.
However, this scheme would fail even when $\dim N = 1$, because Perron's method as usually formulated gives a viscosity solution, but viscosity solutions of (\ref{eikonal equation}) do not have to satisfy $u|_\Gamma = f$ in the classical sense \cite[\S7.C]{Crandall92}.

So we need a strategy for constructing \emph{nonviscosity} solutions of (\ref{eikonal equation}).
If we take $\tilde u := f$, then $\tilde u$ is a subsolution of (\ref{eikonal equation}) in the sense that (classically)
$$
\begin{cases}
    L\tilde u \leq \psi &\text{ on } M \setminus \Gamma,\\
    \tilde u = f &\text{ on } \Gamma.
\end{cases}
$$
The idea, as in Perron's method, is to iteratively modify $\tilde u$ in order to increase $L \tilde u$, without ever violating that $\tilde u$ is a subsolution.
We have to carry out this process carefully, however, because our modifications are artificial oscillations, which could experience destructive interference in the limit.

The trick is to observe that, since $Lu(x)$ (as opposed to, say, the Frobenius norm $|\dif u(x)|_2$) is determined by the behavior of $u$ along small line segments near $x$, and a line segment has positive codimension, we can choose line segments so that every point of $U$ is suitably close to a line segment, while ensuring that all line segments do not interact with each other.
Then our modifications of $\tilde u$ will only be along the line segments.

\subsection{Acknowledgements}
The authors would like to thank Brian Freidin for helpful discussions.

AB was supported by the National Science Foundation's Graduate Research Fellowship Program under Grant No. DGE-2040433.

\section{Preliminaries}
\subsection{Metric geometry}
Throughout, we fix complete Riemannian manifolds $M, N$, and assume that $\dim M \geq 2$.
We denote by $B(x, r)$ the ball of radius $r$ centered on $x \in M$.
If we want to consider balls in another metric space $X$, we write $B_X(x, r)$.
If $B := B(x, r)$, then $\Phi(B) := B(x, r^2)$.

If $X$ is a metric space, $x \in X$, and $Y \subseteq X$ is a set, then 
$$\dist(x, Y) := \inf_{y \in Y} \dist(x, y).$$
If $Z \subseteq X$, then we put 
$$\dist(Y, Z) := \inf_{y \in Y} \dist(y, Z),$$
which we carefully note is not the Hausdorff distance.

\begin{definition}
Let $W \subseteq M$ be an open set and $r > 0$.
An \dfn{$r$-packing} of $W$ is a set $\mathscr P$ of disjoint balls $B(x, r) \Subset W$, such that if $B(x, r) \in \mathscr P$ then the injectivity radius of $M$ at $x$ is at most $r$.
\end{definition}

In particular, if $B(x, r)$ is contained in an $r$-packing, then $\dist(x, \partial W) \geq r$.
An $r$-packing $\mathscr P$ of $W$ is \dfn{maximal} if there does not exist an $r$-packing $\mathscr Q$ of $W$ with $\mathscr Q \supset \mathscr P$.
Every packing is contained in a maximal packing, by Zorn's lemma.

Let $d := \dim M$.
The \dfn{$d - 1$-dimensional upper Minkowski content} of a set $X \subseteq M$ is 
$$\mathcal M^{d - 1}(X) := \limsup_{r \to 0} \frac{\vol(\{x \in M: \dist(x, X) < r\})}{2r}.$$
In particular, if $X$ is a smooth submanifold of $M$ of dimension $d - 1$, then $\mathcal M^{d - 1}(M)$ is the surface area of $M$.

\subsection{Maxima of upper-semicontinous functions}
Let $S$ be a compact Hausdorff space.
If $f: S \to \RR$ is an upper-semicontinuous function, then $f$ attains its maximum somewhere on $S$.
We want to show that we can compute the maximum of $f$ by testing against finite subsets of $S$.

We refer the reader to \cite{kelley1975general} for a review of nets in point-set topology.
The set $\Fin(S)$ of finite subsets of $S$ is a directed set under $\subseteq$, so we can think of a function $\Theta: \Fin(S) \to \RR$ as a net.
By definition, the equation
$$L = \lim_{X \in \Fin(S)} \Theta(X)$$
means that for every $\varepsilon > 0$ there exists $X \in \Fin(S)$ such that if $Y \in \Fin(S)$ and $X \subseteq Y$, then
$$|\Theta(Y) - L| < \varepsilon.$$
From this definition it follows that for any upper-semicontinuous function $f: S \to \RR$,
\begin{equation}\label{maximum is limit of finite sets}
    \max_{x \in S} f(x) = \lim_{X \in \Fin(S)} \max_{x \in X} f(x).
\end{equation}

\subsection{Lipschitz maps}
Let $X$ be a metric space, and $u: X \to N$ a Lipschitz map.
Then the \dfn{local Lipschitz constant} is 
$$Lu(x) := \limsup_{r \to 0} \Lip(u, B_X(x, r)).$$
Since $\Lip(u, B_X(x, r))$ is a nondecreasing function of $r$, this limit superior is actually a limit and an infimum, and the function $Lu$ is upper-semicontinuous.
It is easily seen for fixed $x$ that $u \mapsto Lu(x)$ is a seminorm on the space of Lipschitz maps.
If $X$ is a smooth manifold, then for almost every $x \in X$, $Lu(x) = |\dif u(x)|_\infty$, where $|\cdot|_\infty$ is the operator norm on matrices, and $\dif u \in \Hom(TX, u^*(TN))$ is the derivative of $u$.
We refer the reader to \cite[\S4]{Crandall2008} for a more careful discussion of the local Lipschitz constant.

\begin{lemma}
Let $K$ be a convex compact subset of $M$ and $u: K \to N$ a Lipschitz map.
Then
\begin{equation}\label{Lip is sup of local Lips}
\Lip(u, K) = \max_{z \in K} Lu(z).
\end{equation}
\end{lemma}

In this lemma, we are taking $Lu$ to be computed in the metric space $K$.
In other words,
$$Lu(z) = \lim_{r \to 0} \Lip(u, B_K(z, r)) = \lim_{r \to 0} \Lip(u, K \cap B(z, r)).$$

\begin{proof}
Clearly for any $x \in K$, $Lu(x) \leq \Lip(u, K)$.
Conversely, let $x, y \in K$.
Since $K$ is convex, there is a minimizing geodesic $\ell \subseteq K$ from $x$ to $y$.
Then $\ell \subseteq K$, and by \cite[Lemma 4.2(d)]{Crandall2008},
\begin{align*}
    \dist_N(u(x), u(y)) \leq \max_{z \in \ell} Lu(z) \dist_M(x, y) \leq \max_{z \in K} Lu(z) \dist_M(x, y).
\end{align*}
The proof of the cited lemma goes through when $Lu(z)$ is computed in the metric space $\ell$, not in the ambient metric space $M$.
Similarly, the proof assumes that $N = \RR$, but the proof goes through generally.
\end{proof}

Let $K > 0$ and let $(u_n)$ be a sequence of maps on a compact metric space $X$, such that $u_n \to u$ pointwise and $\Lip(u_n, X) \leq K$.
Then it follows from the Arzela-Ascoli theorem that $u_n \to u$ uniformly and 
\begin{equation}\label{bound on limit Lip}
    \Lip(u, X) \leq K.
\end{equation}
In fact, we can do better, and show that $L$ is lower-semicontinuous:

\begin{lemma}\label{Lip is lower semicontinuous}
Suppose that $u_k, u: M \to N$ are maps with $u_k \to u$ in $L^1$ and $\Lip(u_k, U) \leq C$.
Then for almost every $x \in M$,
$$Lu(x) \leq \liminf_{k \to \infty} Lu_k(x).$$
\end{lemma}
\begin{proof}
We first use the Lebesgue differentiation theorem to write, for almost every $x \in M$,
$$Lu(x) = \lim_{r \to 0} \dashint_{B(x, r)} Lu(y) \dif y = \lim_{r \to 0} \dashint_{B(x, r)} |\dif u(y)|_\infty \dif y.$$
We now use lower-semicontinuity of $|\cdot|_\infty$-total variation to bound 
$$\dashint_{B(x, r)} |\dif u(y)|_\infty \dif y \leq \liminf_{k \to \infty} \dashint_{B(x, r)} |\dif u_k(y)|_\infty \dif y = \liminf_{k \to \infty} \dashint_{B(x, r)} Lu_k(y) \dif y.$$
The lower-semicontinuity of $|\cdot|_\infty$-total variation can be proven, analogously to the scalar case \cite[Theorem 1.9]{Giusti1984minimal}, using the duality between the norms $\||\cdot|_\infty\|_{L^1}$ and $\||\cdot|_1\|_{L^\infty}$ on matrix-valued functions (where $|\cdot|_1$ is the trace norm), and we omit the details.

Because $Lu_k(y) \leq \Lip(u_k, U) \leq C$, we can apply dominated convergence, and then use the Lebesgue differentiation theorem again: for almost every $x \in M$,
\begin{align*}
    \lim_{r \to 0} \liminf_{k \to \infty} \dashint_{B(x, r)} Lu_k(y) \dif y &= \lim_{r \to 0} \dashint_{B(x, r)} \liminf_{k \to \infty} Lu_k(y) \dif y = \liminf_{k \to \infty} Lu_k(x). \qedhere
\end{align*}
\end{proof}

Let $u$ be a Lipschitz map.
Though $Lu$ is not lower-semicontinuous, it satisfies the following ``averaged" version of lower-semicontinuity, which was implicit in \cite[\S1.4.2]{Sheffield12}:

\begin{lemma}\label{lower semicontinuity on average}
Let $u: M \to N$ be a Lipschitz map, and let $\psi: M \to \RR_+$ a lower-semicontinuous function.
If $Lu \leq \psi$ on almost all of $M$, then $Lu \leq \psi$ on $M$.
\end{lemma}
\begin{proof}
Let $\varepsilon > 0$ and $x \in M$.
Then there is an open neighborhood $V$ of $x$ such that:
\begin{enumerate}
    \item If $y \in V$, then $\psi(y) \leq \psi(x) + \varepsilon$.
    \item The image $u_*(V)$ is contained in a convex subset of $N$.
\end{enumerate}
Let $\Delta \subset V^2$ be the diagonal over $V$.
Then the difference quotient
$$f(y, z) := \frac{\dist_N(u(x), u(y))}{\dist_M(x, y)}$$
satisfies
$$Lu(x) \leq \Lip(u, V) \leq \sup_{y, z \in V^2 \setminus \Delta} f(y, z). $$
Since $\Delta$ is a null subset of $V^2$ and $f$ is continuous on $V^2 \setminus \Delta$, this supremum is the $L^\infty$ norm of $f$.

Given $y, z \in V$, let $[y, z]$ be a minimizing geodesic from $y$ to $z$.
We estimate using the mean value and Jensen inequalities
\begin{align*}
    \|f\|_{L^\infty}
    &= \lim_{p \to \infty} \|f\|_{L^p} \\
    &\leq \lim_{p \to \infty} \left[\iint_{V^2} \frac{1}{\dist_M(y, z)^p} \int_{[y, z]} Lu(w)^p \dif w \dif y \dif z\right]^{1/p} \\
    &\leq \lim_{p \to \infty} \left[\iint_{V^2} \dashint_{[y, z]} Lu(w)^p \dif w \dif y \dif z\right]^{1/p}.
\end{align*}
Let $E(y, w)$ be the set of points in $V$ on the line $\ell$ through $y, w$ which are to the right of $w$, where $\ell$ is oriented so that the vector $w - y$ points to the right.
Then $w \in [y, z]$ iff $z \in E(y, w)$, so we can rewrite the bounds of integration using Fubini's theorem as follows:
$$\iint_{V^2} \dashint_{[y, z]} Lu(w)^p \dif w \dif y \dif z = \iint_{V^2} \int_{E(y, w)} Lu(w)^p \frac{\dif z}{\dist_M(y, z)} \dif y \dif w.$$
The point is that $w$ now ranges over $V$, and we know that for almost every $w \in V$,
$$Lu(w) \leq \psi(x) + \varepsilon.$$
We therefore can estimate this integral, and undo the change in the order of integration:
\begin{align*} 
\iint_{V^2} \int_{E(y, w)} Lu(w)^p \frac{\dif z}{\dist_M(y, z)} \dif y \dif w
&\leq (\psi(x) + \varepsilon)^p \iint_{V^2} \int_{E(y, w)} \frac{\dif z}{\dist_M(y, z)} \dif y \dif w \\
&= (\psi(x) + \varepsilon)^p \iint_{V^2} \dashint_{[y, z]} \dif w \dif y \dif z \\
&= (\psi(x) + \varepsilon)^p \vol(V)^2.
\end{align*}
In other words, we can estimate for any $\varepsilon > 0$,
\begin{align*}
    Lu(x) &\leq \|f\|_{L^\infty} \leq (\psi(x) + \varepsilon) \lim_{p \to \infty} \vol(V)^{2/p} = \psi(x) + \varepsilon. \qedhere
\end{align*}
\end{proof}

Finally, we observe that for functions on $\RR$, the Lipschitz constant is preserved by taking maxima:

\begin{lemma}\label{Lip of max}
Let $I \subseteq \RR$ be an interval and $f, g: I \to \RR$ Lipschitz functions.
Then 
$$\Lip(\max(f, g), I) \leq \max(\Lip(f, I), \Lip(g, I)).$$
\end{lemma}
\begin{proof}
Let $h := \max(f, g)$, $K := \max(\Lip(f, I), \Lip(g, I))$, and $x, y \in I$.
We may assume that $x < y$.
We have a suitable estimate on $|h(x) - h(y)|$ if both $f(x) \geq g(x)$ and $f(y) \geq g(y)$, or if $g(x) \geq f(x)$ and $g(y) \geq f(y)$.
So we may assume that $f(x) > g(x)$ and $g(y) > f(y)$.
Then there exists $z \in (x, y)$ such that $f(z) = g(z)$, and we estimate 
\begin{align*}
    |h(x) - h(y)| &= |f(x) - g(y)| \leq |f(x) - f(z)| + |g(z) - g(y)| \\
    &\leq K|x - z| + K|z - y| = K|x - y|. \qedhere
\end{align*}
\end{proof}

\section{Proof of Theorem \ref{eikonal extension}}
\subsection{Modifying the Lipschitz constant on a packing}
We want to show that given a packing $\mathscr P$ and a Lipschitz map $u$, we can modify $u$ along a line segment close to the center of each ball in $\mathscr P$, while ensuring that $u$ is held fixed away from the centers of the balls in $\mathscr P$.
Recalling our notation $\Phi(B(x, r)) := B(x, r^2)$, our goal in this section is to prove:

\begin{proposition}\label{improvement on packing}
Let $V \subseteq M$ be an open set, $0 < r \leq 1$, and $\psi: V \to \RR_+$ a continuous function.
Let $u: V \to N$ satisfy $Lu \leq \psi$.
Then for every $r$-packing $\mathscr P$ of $V$, there exists $v: V \to N$ such that:
\begin{enumerate}
\item For every $x \in V$, $Lv(x) \leq \psi(x)$.
\item For every $x \in V \setminus \bigcup \{\Phi(B): B \in \mathscr P\}$, $v(x) = u(x)$.
\item For every $\varepsilon > 0$ and every $B \in \mathscr P$, there exists a closed line segment $\ell \Subset \Phi(B)$ such that 
$$\Lip(v, \ell) > \max_{x \in \ell} \psi(x) - \varepsilon.$$
\item $v$ is homotopic to $u$ relative to $\partial V$.
\end{enumerate}
\end{proposition}

We are going to work our way up to the proof of Proposition \ref{improvement on packing}.
First, we show how to improve $Lu$ at a single point in a ball $B$ without modifying $u$ away from $B$.
We use this to show to modify $u$ along a line segment inside a prescribed open set $W$ without modifying $u$ away from $W$.
By taking $W = \Phi(B)$ where $B \in \mathscr P$, we deduce Proposition \ref{improvement on packing}.

On Riemannian manifolds, we need the following lemma, which is essentially obvious on euclidean space:

\begin{lemma}\label{existence of oscillators}
Let $S \subseteq M$ be a ball, $D \geq 1$, and let $u: M \to N$ be a continuous map.
There exists a family of smooth maps $\chi_t: M \to \RR^D$ such that for each $t \geq 0$,
\begin{enumerate}
    \item $\Lip(\chi_t, S) \geq t$.
    \item $\supp \chi_t \subseteq S$.
    \item For each $s \geq 0$,
    \begin{equation}\label{continuous dependence in t}\Lip(\chi_t - \chi_s, S) \lesssim |t - s|
    \end{equation}
    where the implied constant only depends on the geometry of $N$ near the image of $S$ under $u$.
    \item For each $z \in S$, $|\chi_t(z)|$ is smaller than the injectivity radius of $N$ at $u(z)$.
\end{enumerate}
\end{lemma}
\begin{proof}
Let $z_0$ be the center of $S$.
By replacing $S$ with a smaller ball centered on $z_0$ if necessary, we may assume that $S$ is the image of a ball in the tangent space $T_{z_0} M$ centered on $0$, such that the exponential map onto $S$ is a diffeomorphism which is close in $C^\infty$ to an isometry.
Furthermore, we may shrink $S$ so that there exists $\varepsilon_0 > 0$ such that for every $z \in S$, the injectivity radius of $N$ at $u(z)$ is at least $\varepsilon_0$.

For $\xi \in \RR^d$, define $f_t(\xi) \in \RR^D$ by
$$f_t(\xi) = \varepsilon_0 \left(\sin \left(\frac{10t\xi_1}{\varepsilon_0}\right), 0, \dots, 0\right).$$
By identifying $\RR^d$ with the tangent space $T_{z_0} M$, and pushing forward using the exponential map, we obtain a map $\tilde f_t$ defined on $S$.
Let $\varphi$ be a smooth, compactly supported function on $S$ which is $1$ near $z_0$, and let $\chi_t := \varphi \tilde f_t$.

It is clear that $\supp \chi_t \subseteq S$ and $|\chi_t(z)| \leq \varepsilon_0$ is smaller than the injectivity radius of $N$ at $u(z)$.
By (\ref{Lip is sup of local Lips}), and the bounds on the exponential map,
$$\Lip(\chi_t, S) \geq L\chi_t(z_0) \geq \frac{\varepsilon_0}{10} \frac{\dif}{\dif \xi_1} \sin\left(\frac{10t\xi_1}{\varepsilon_0}\right)\bigg|_{\xi_1 = 0} = t.$$

Next we use the bounds on the exponential map to bound
\begin{align*}
    \left|\frac{\dif}{\dif t} \nabla \chi_t(\xi)\right|
    &= |\varphi(\xi)| \left|\frac{\dif}{\dif t} \nabla \tilde f_t(\xi)\right| 
    \leq \frac{20}{\varepsilon_0} |\varphi(\xi)| \left|\frac{\dif}{\dif t} \cos\left(\frac{10 t\xi_1}{\varepsilon}\right)\right| 
    \leq \frac{200}{\varepsilon_0^2}
\end{align*}
(where we use the fact that $|\xi_1| \leq \varepsilon_0$).
Since this bound does not depend on $t$, we conclude (\ref{continuous dependence in t}) from the mean value theorem.
\end{proof}

\begin{lemma}\label{improvement at one point}
Suppose that $B(x, \delta) \Subset V$, and let $\psi: V \to \RR_+$ be a lower-semicontinuous function.
Let $u: V \to N$ satisfy $Lu \leq \psi$.
Then there exists $v: V \to N$ such that:
\begin{enumerate}
    \item For every $y \in V$, $Lv(y) \leq \psi(y)$.
    \item For every $y \in V \setminus B(x, \delta)$, $u(y) = v(y)$.
    \item There exists $z \in B(x, \delta)$ such that $Lv(z) = \psi(z)$.
    \item $v$ is homotopic to $u$ relative to $\partial V$.
\end{enumerate}
\end{lemma}
\begin{proof}
We may assume that $Lu < \psi$ on $B(x, \delta)$, or else there is nothing to show.
Let $S$ be the closure of $B(x, \delta/2)$, and let $D := \dim N$.
Let $\chi_t$ be as in Lemma \ref{existence of oscillators} applied to $S$ and $D$.
For each $z \in V$ and $t \geq 0$, we define $u_t(z)$ in normal coordinates centered at $u(z)$ by
$$u_t(z) := u(z) + \chi_t(z).$$
In other words, we identify the tangent space $T_{u(z)} N$ with $\RR^D$, and then $u_t(z)$ is the image of $\chi_t(z)$ under the exponential map $T_{u(z)} N \to N$.

\begin{claim}
    Let $$h(t) := \max_{z \in S} \left(Lu_t(z) - \psi(z)\right).$$
    Then $h$ is continuous.
\end{claim}
\begin{proof}[Proof of claim]
We first remark that $h$ is well-defined, since $Lu_t$ is upper-semicontinuous, $\psi$ is continuous, and $S$ is compact, so that the maximum is actually attained.

Let $z \in V$.
We can embed the space of maps $M \to B_N(u(z), \varepsilon)$, where $\varepsilon$ is smaller than the injectivity radius of $N$ at $u(z)$, in the space of maps $M \to T_{u(z)} N$.
Then we can think of $w \mapsto Lw(z)$ as a seminorm, and use the reverse triangle inequality to bound
$$|Lu_t(z) - Lu_s(z)| \leq |L(u_t - u_s)(z)| = |L(\chi_t - \chi_s)(z)|.$$
We then estimate using (\ref{Lip is sup of local Lips}) and the fact that $\Lip(\chi_t - \chi_s, M) \lesssim t - s$ that
\begin{equation}
    |Lu_t(z) - Lu_s(z)| \lesssim |t - s|. \label{base case of finite set induction}
\end{equation}

It is enough to show that for every $s > 0$, $h$ is Lipschitz continuous on $[0, s]$.
To this end, we claim that there exists $C > 0$ such that for every finite set $X \in \Fin(S)$,
\begin{equation}\label{conclusion of finite set induction}
\Lip\left(t \mapsto \max_{z \in X} \left(Lu_t(z) - \psi(z)\right), [0, s]\right) \leq C.
\end{equation}
We prove this by induction on $\card X$.
The base case $X := \{z\}$ is given by (\ref{base case of finite set induction}).
The inductive case is given by Lemma \ref{Lip of max}.

By (\ref{maximum is limit of finite sets}), for every $t \in \RR$,
$$h(t) = \lim_{X \in \mathrm{Fin}(S)} \max_{z \in X} \left(Lu_t(z) - \psi(z)\right),$$
so we conclude from (\ref{conclusion of finite set induction}) and (\ref{bound on limit Lip}), and the compactness of $[0, s]$, that
\begin{align*}
    \Lip(h, [0, s]) &\leq C. \qedhere
\end{align*}
\end{proof}

By assumption, $h(t) < 0$.
On the other hand, $h(t) \to \infty$ as $t \to \infty$.
So by the claim, there exists $T$ with $h(T) = 0$.
Let $v := u_T$ and let $z$ be the point which realizes the maximum in the definition of $h(T)$.
Then:
\begin{enumerate}
    \item For every $y \in V$,
    $$Lv(y) \leq h(T) + \psi(y) = \psi(y).$$
    \item Since $\chi_t$ is supported in $S$, which is a compact subset of $B(x, \delta)$ $u = v$ away from $B(x, \delta)$.
    \item By definition of $z$,
    $$Lv(z) = h(T) + \psi(z) = \psi(z).$$
    \item $(u_t)_{t \in [0, T]}$ is a homotopy between $u_0$ and $v$, which is a homotopy relative to the boundary since $u_t - u = \chi_t$ has compact support in $B(x, \delta)$.
\end{enumerate}
So $v$ is the desired map, witnessed by the point $z$.
\end{proof}

\begin{lemma}\label{improvement on one ball}
Let $W \subseteq V$ be an open set and $\psi: V \to \RR_+$ a continuous function.
Let $u: V \to N$ satisfy $Lu \leq \psi$.
Then there exists $v: V \to N$ such that:
\begin{enumerate}
\item For every $x \in V$, $Lv(x) \leq \psi(x)$. \label{one ball bound}
\item For every $x \in V \setminus W$, $v(x) = u(x)$. \label{one ball preservation}
\item For every $\varepsilon > 0$, there exists a closed minimizing geodesic segment $\ell \Subset W$ such that \label{one ball improvement}
$$\Lip(v, \ell) > \max_{x \in \ell} \psi(x) - \varepsilon.$$
\item $v$ is homotopic to $u$ relative to $\partial V$. \label{one ball homotopy}
\end{enumerate}
\end{lemma}
\begin{proof}
Choose a ball $W' \Subset W$, and apply Lemma \ref{improvement at one point} to the ball $W'$ to obtain a map $v$ satisfying (\ref{one ball bound}), (\ref{one ball preservation}), and (\ref{one ball homotopy}), and a point $z \in W'$ such that $Lv(z) = \psi(z)$.
For each $\rho > 0$, let $\mathscr L_\rho$ be the set of closed line segments in $B(z, \rho)$.
To prove (\ref{one ball improvement}), we are going to apply compactness and contradiction in the limit $\rho \to 0$.

If (\ref{one ball improvement}) fails, then there exists $\varepsilon > 0$ such that for every $\rho > 0$ and every $\ell \in \mathscr L_\rho$,
$$\Lip(v, \ell) \leq \max_{x \in \ell} \psi(x) - \varepsilon.$$
So, since $B(z, \rho)$ is convex,
$$\Lip(v, B(z, \rho)) = \sup_{\ell \in \mathscr L_\rho} \Lip(v, \ell) \leq \sup_{x \in B(z, \rho)} \psi(x) - \varepsilon.$$
Since $\psi$ is upper-semicontinuous, and the limit in the definition of $Lv(z)$ is actually an infimum, it follows that 
$$Lv(z) \leq \Lip(v, B(z, \rho)) \leq \psi(z) - \varepsilon,$$
which is a contradiction since $Lv(z) = \psi(z)$.
\end{proof}

\begin{proof}[Proof of Proposition \ref{improvement on packing}]
Let $u_0 := u$.
Since $\mathscr P$ consists of disjoint open subsets of $W$, $\mathscr P$ is countable.
We enumerate it as
$$\mathscr P = \{B_i: i < \card \mathscr P\}.$$
If $i < \card \mathscr P$ and we are given $u_i$, let $u_{i + 1}$ be the map returned by Lemma \ref{improvement on one ball} when it is given the function $u_i$ and the ball $W_i := \Phi(B_i)$ as input.

If $\mathscr P$ is finite and $i := \card \mathscr P$, then the desired map is $u := u_i$.
Otherwise, for each fixed $i \in \NN$, $u_j|_{W_i}$ stabilizes when $j = i$.
Moreover, $u_j = u_0$ on $V \setminus \bigcup_i W_i$.
Therefore there is a map $u$ such that $u_j \to u$ locally uniformly, and $u$ has all desired properties.
\end{proof}

\subsection{Iterating the packing procedure}
We are going to iteratively apply Proposition \ref{improvement on packing} to finer and finer packings; at each dyadic scale, we get some improvement to a given Lipschitz map $u_0$.
Later we shall take a limit as the scale goes to $0$; it is important that we choose our packings so carefully that in the limit, our improvements do not cancel each other.

\begin{lemma}\label{iterating the packing}
Let $V \subseteq M$ be an open set such that $\vol(V) < \infty$, and $\psi: V \to \RR_+$ a continuous function. 
Let $u_0 : V \to N$ satisfy $Lu_0 \leq \psi$.
Let $\mathscr L_0 := \emptyset$.
Then there exist
\begin{enumerate}
\item maps $u_i: V \to N$ such that $Lu_i \leq \psi$,
\item finite sets $\mathscr L_i \supseteq \mathscr L_{i - 1}$, whose elements are closed minimizing geodesic segments,
\item and maximal $2^{-i}$-packings $\mathscr P_i$ of $V \setminus \bigcup \mathscr L_{i - 1}$,
\end{enumerate}
such that for every $i \geq 1$:
\begin{enumerate}
\item \label{balls and lines} There is a bijection
$$\ell: \mathscr P_i \to \mathscr L_i \setminus \mathscr L_{i - 1},$$
which sends $B \in \mathscr P_i$ to a line segment $\ell(B) \Subset \Phi(B)$ such that
$$\Lip(u_i, \ell(B)) > \max_{x \in \ell(B)} \psi(x) - 2^{-i}.$$
\item \label{preservation} For each $x \in V \setminus \bigcup \{\Phi(B): B \in \mathscr P_i\}$,
$$u_i(x) = u_{i - 1}(x).$$
\item \label{homotopy} $u_i$ is homotopic to $u_0$ relative to $\partial V$.
\end{enumerate}
\end{lemma}
\begin{proof}
We proceed by induction. Assume that we have defined $u_{i - 1}, \mathscr L_{i - 1}, \mathscr P_{i - 1}$ to have the desired properties.
Let $W := V \setminus \bigcup \mathscr L_{i - 1}$.
Then $W$ is an open set since $\bigcup \mathscr L_{i - 1}$ is a finite union of closed sets.
So there is a maximal $2^{-i}$-packing $\mathscr P_i$ of $W$, which is moreover finite since $\vol(W) = \vol(V) < \infty$.
Let
$$Z := \bigcup \{\Phi(B): B \in \mathscr P_i\}.$$
By Proposition \ref{improvement on packing}, we can find $u_i: W \to \RR^D$ and line segments $\ell(B)$, $B \in \mathscr P_i$, such that $u_i$ agrees with $u_{i - 1}$ away from $Z$ and satisfies $Lu_i \leq K$, but for each $B \in \mathscr P_i$,
$$\Lip(u_i, \ell(B)) > \max_{x \in \ell(B)} \psi(x) - 2^{-i}.$$
Moreover, $u_i$ is homotopic to $u_{i - 1}$ relative to $\partial V$.

Since $V \setminus W$ does not intersect $Z$, $u_i = u_{i - 1}$ near $W$, and in particular, extends uniquely to all of $V$ while satisfying (\ref{preservation}).
Let
$$\mathscr L_i := \mathscr L_{i - 1} \cup \{\ell(B): B \in \mathscr P_i\}.$$
Then $\mathscr L_i$ satisfies (\ref{balls and lines}) by construction, and is finite since $\mathscr L_{i - 1}$ and $\mathscr P_i$ are finite.
Since $u_{i - 1}$ is homotopic to $u_0$ by induction, and $u_i$ is homotopic to $u_{i - 1}$, (\ref{homotopy}) follows.
\end{proof}

Throughout the rest of this section, we fix a Lipschitz map $u_0: V \to N$, and let $\mathscr L_i, \mathscr P_i$ be as in Lemma \ref{iterating the packing}.
We will later show that for every dyadic scale $2^{-i}$ with $i$ large enough, and every $x \in V$, we can estimate $|Lu_i(x) - \psi(x)|$ provided that there exists $j \sim i$ and $B \in \mathscr P_j$ such that $x$ is $2^{-i}$-close to the center of $B$.
Thus we have two goals:
\begin{enumerate}
    \item We want to show that for every $x$ which is $2^{-i}$-far from $\partial V$, there exists $j \leq i$ and $B \in \mathscr P_j$ such that $x$ is $2^{-i}$-close to the center of $B$. This is accomplished by Lemma \ref{almost density of balls}.
    \item We want to show that the set of points which are either $2^{-i}$-close to $\partial V$, or are $2^{-i}$-close to the center of a ball $B \in \mathscr P_j$ with $j \leq i/2$, is exponentially small. This is accomplished by Lemma \ref{decay of the exceptional set}.
\end{enumerate}
The proof of Lemma \ref{almost density of balls} requires a lemma on the separation of the line segments, Lemma \ref{separation of line segments}.
Moreover, Lemmata \ref{almost density of balls} and \ref{decay of the exceptional set} only hold under certain geometric assumptions on the underlying manifold, given by Assumption \ref{geometric assumptions}; we will be able to remove these assumptions later.

\begin{lemma}\label{separation of line segments}
For every $\ell, \ell' \in \mathscr L_i$, either $\ell = \ell'$ or
$$\dist(\ell, \ell') > 2^{-i} - 4^{-i}.$$
\end{lemma}
\begin{proof}
Suppose that $\ell \neq \ell'$.
By induction, we may assume that this result holds if $\ell, \ell' \in \mathscr L_{i - 1}$.
So suppose that $\ell \in \mathscr L_i \setminus \mathscr L_{i - 1}$, so by Lemma \ref{iterating the packing}(\ref{balls and lines}) there exists $B \in \mathscr P_i$ such that $\ell \Subset \Phi(B)$.

We claim that $\ell'$ does not intersect $B$. This follows from casework:
\begin{enumerate}
    \item If $\ell' \in \mathscr L_{i - 1}$, then since balls in $\mathscr P_i$ are disjoint from $\bigcup \mathscr L_{i - 1}$, $\ell'$ does not intersect $B$.
    \item If $\ell' \in \mathscr L_i \setminus \mathscr L_{i - 1}$, then by Lemma \ref{iterating the packing}(\ref{balls and lines}), there exists $B' \in \mathscr P_i$ such that $\ell' \subset \Phi(B') \Subset B'$. Since $\mathscr P_i$ is a packing, $\ell'$ does not intersect $B$.
\end{enumerate}
In either case, the result follows, because
\begin{align*} 
\dist(\ell, \ell') &\geq \dist(\ell, V \setminus B) > \dist(\Phi(B), V \setminus B) > 2^{-i} - 4^{-i}. \qedhere
\end{align*}
\end{proof}

\begin{assumption}\label{geometric assumptions}
We assume that:
\begin{enumerate}
    \item The Lebesgue measure $\vol(V)$ is finite.
    \item Let $d := \dim M$. Then the Minkowski content $\mathcal M^{d - 1}(\partial V)$ is finite.
    \item The curvature tensor $\Riem_M$ satisfies
    $$\|\Riem_M\|_{C^0} \leq \varepsilon_0,$$
    where $\varepsilon_0 > 0$ is a small dimensional constant to be chosen later.
\end{enumerate}
\end{assumption}

\begin{lemma}\label{almost density of balls}
Suppose that Assumption \ref{geometric assumptions} holds. Let
$$\mathscr U_i := \{x \in V: B(x, 2^{-i}) \in \mathscr P_i\}.$$
For every $i \geq 10$ and $x \in V$,
$$\min\left(\dist(x, \partial V), \min_{j \leq i} \dist(x, \mathscr U_j)\right) \leq 2^{-i + 4}.$$
\end{lemma}
\begin{proof}
Suppose not, so that some ball $B(x, 2^{-i + 4}) \Subset V$ does not contain any point of $\bigcup_{j \leq i} \mathscr U_j$.
We first claim:

\begin{claim}
    Every ball $B(y, 2^{-i + 1}) \subseteq B(x, 2^{-i + 4})$ intersects a geodesic segment in $\mathscr L_{i - 1}$.
\end{claim}
\begin{proof}[Proof of claim]
    Otherwise, by maximality of the packing $\mathscr P_i$, and the completeness of $M$, any point of $V \setminus \bigcup \mathscr L_{i - 1}$ must be within $2^{-i}$ of a ball in $\mathscr P_i$, or $\partial V$, or $\bigcup \mathscr L_{i - 1}$.
But by definition of $y$, no ball in $\mathscr P_i$ intersects $B(y, 2^{-i})$, a set which also does not intersect $\partial V$ or $\mathscr L_{i - 1}$.
This is a contradiction.
\end{proof}

Applying the claim with $y := x$, we can find $j_1 \leq i - 1$ and $q_1 \in \mathscr U_{j_1}$ such that $\ell_1 := \ell(B(q_1, 2^{-j_1}))$ intersects $B(x, 2^{-i + 1})$.
Since $q_1 \notin B(x, 2^{-i + 4})$, but $\ell_1 \subset \Phi(B(q_1, 2^{-j_1})) = B(q_1, 4^{-{j_1}})$,
$$\dist(\partial B(x, 2^{-i + 4}), B(x, 2^{-i + 1})) < 4^{-j_1}.$$
From this it follows that
$$2^{-2j_1} = 4^{-j_1} > 2^{-i + 4} - 2^{-i + 1} > 2^{-i + 3}.$$
Therefore $j_1 < (i - 3)/2$.

Since $\dim M \geq 2$, there is a ball $B(y, 2^{-i + 1}) \subseteq B(x, 2^{-i + 2})$ which does not intersect $\ell_1$.
So by the claim, there exist $j_2 \leq i - 1$ and $q_2 \in \mathscr U_{j_2}$ such that $\ell_2 := \ell(B(q_2, 2^{-j_2}))$ intersects $B(y, 2^{-i + 1})$.
Arguing as above, we see that $j_2 \leq (i - 3)/2$ as well.

Since $\ell_1$ and $\ell_2$ both intersect $B(x, 2^{-i + 2})$,
$$\dist(\ell_1, \ell_2) < 2^{-i + 3}.$$
On the other hand, by Lemma \ref{separation of line segments},
$$\dist(\ell_1, \ell_2) > 2^{-\frac{i - 3}{2}} - 4^{-\frac{i - 3}{2}},$$
which is a contradiction when $i \geq 10$.
\end{proof}

\begin{lemma}\label{decay of the exceptional set}
Suppose that Assumption \ref{geometric assumptions} holds, where the curvature scale $\varepsilon_0$ is chosen sufficiently small.
Introduce the exceptional set
\begin{equation}\label{exponential exceptional set}
Z_i := \bigcup \left\{B(x, 2^{-i+4}): x \in \left(\partial V \cup \bigcup_{j \leq i/2} \mathscr U_j\right)\right\}.
\end{equation}
Then there exists $I \geq 10$ which only depends on the geometry of $M$ and $V$, such that if $i \geq I$ then
$$\vol(Z_i) \leq \frac{64 \mathcal M^{d - 1}(\partial V) + 64^d \vol(V)}{2^i}.$$
\end{lemma}
\begin{proof}
Let 
$$T_i := \{z \in V: \dist(z, \partial V) < 2^{-i + 4}\}.$$
By definition of the upper Minkowski content, there exists $I \geq 10$ such that if $i \geq I$,
$$\mathcal M^{d - 1}(\partial V) \geq \frac{\vol(T_i)}{2^{-i + 6}}.$$

Let $\alpha$ be the volume of the unit ball in $\RR^d$.
If $\varepsilon_0$ is small enough, then for every $x \in M$ with $\dist(x, \partial V) > 2^{-j}$,
$$2^{-(j - 1)d}\alpha \leq \vol(B(x, 2^{-j})) \leq 2^{-(j + 1)d}\alpha.$$
Since $\mathscr P_j$ is a set of disjoint $2^{-j}$-balls in $V$, and is in bijection with $\mathscr U_j$,
$$\card \mathscr U_j \leq \frac{\vol(V)}{\inf_{x \in \mathscr U_j} \vol(B(x, 2^{-j}))} \leq \frac{\vol(V)}{\alpha} 2^{(j + 1)d}.$$
Therefore
$$\card \left(\bigcup_{j \leq i/2} \mathscr U_j\right) \leq \sum_{j=1}^{i/2} \card \mathscr U_j \leq \frac{\vol(V)}{\alpha} \sum_{j=1}^{i/2} 2^{(j + 1) d} \leq \frac{2^d \vol(V)}{\alpha} 2^{id/2}.$$
Therefore, since $d \geq 2$,
\begin{align*}
    \vol(Z_i)
    &\leq \vol(T_i) + \vol(B(0, 2^{-i + 4})) \card \left(\bigcup_{j \leq i/2} \mathscr U_j\right) \\
    &\leq \frac{64 \mathcal M^{d - 1}(\partial V)}{2^i} + \frac{64^d \vol(V)}{2^{id/2}} \\
    &\leq \frac{64 \mathcal M^{d - 1}(\partial V) + 64^d \vol(V)}{2^i}. \qedhere
\end{align*}
\end{proof}

\subsection{Taking the limit of the packings}
We are now ready to prove the main theorem under the geometric hypothesis, Assumption \ref{geometric assumptions}.

\begin{lemma}\label{smooth case}
Theorem \ref{eikonal extension} holds if we in addition assume for $V := M \setminus \Gamma$ that Assumption \ref{geometric assumptions} holds.
\end{lemma}
\begin{proof}
Let $u_0 := f$, let $u_i, \mathscr P_i, \mathscr L_i$ be as in Lemma \ref{iterating the packing}, and let $\mathscr U_i$ be as in Lemma \ref{almost density of balls}.
Since $Lu_i \leq \psi$, and $\psi$ is continuous, we obtain from (\ref{Lip is sup of local Lips}) that $u_i$ is Lipschitz on each ball $B(0, R)$.
So by a diagonal argument and the Arzela-Ascoli theorem, after taking a subsequence there exists a locally Lipschitz map $u$ such that $u_i \to u$ locally uniformly.
Then $u$ is homotopic to $f$, and by Lemma \ref{Lip is lower semicontinuous}, $Lu \leq \psi$ on almost all of $V$, so by Lemma \ref{lower semicontinuity on average}, $Lu \leq \psi$ on all of $V$.

It remins to show that $Lu = \psi$ everywhere.
To this end, let $Z_i$ and $I$ be as in Lemma \ref{decay of the exceptional set}.
Introduce the exceptional set
$$Z := \bigcap_{i \geq I} \bigcup_{j \geq i} Z_j.$$
By Lemma \ref{decay of the exceptional set}, there exists $C > 0$ which only depends on the geometry of $V$ such that
$$\vol\left(\bigcup_{j \geq i} Z_j\right) \leq \frac{C}{2} \sum_{j=i}^\infty 2^{-j} = \frac{C}{2^i}.$$
By continuity from above, it follows that $Z$ is null.

For every $x \in V \setminus Z$ and $\varepsilon > 0$, there exists $i \geq I$ such that
$$2^{-i} < \max(\varepsilon^2, 16 \dist(x, \partial V))$$
and $x \notin \bigcup_{j \geq i} Z_j$.
By Lemma \ref{almost density of balls}, there exists $j \geq i/2$ and $q \in \mathscr U_j$ such that $|x - q| < 2^{-i + 4}$.

From Lemma \ref{iterating the packing}(\ref{balls and lines}), there exists a geodesic segment $\ell \in \mathscr L_j$ such that
\begin{align*}
\ell &\Subset B(q, 4^{-j}) \subseteq B(q, 2^{-i}) \\
\Lip(u_j, \ell) &> \max_{z \in \ell} \psi(z) - 2^{-j} \geq \max_{z \in \ell} \psi(z) - 2^{-i/2} > \max_{z \in \ell} \psi(z) - \varepsilon.
\end{align*}
An induction on Lemma \ref{iterating the packing}(\ref{preservation}) implies that for $p \geq j$, $u_p|_\ell = u_j|_\ell$; taking the limit, we see that $u|_\ell = u_j|_\ell$.
Therefore
$$\Lip(u, \ell) > \max_{z \in \ell} \psi(z) - \varepsilon,$$
so by (\ref{Lip is sup of local Lips}), there exists $y_i \in \ell$ such that
$$Lu(y_i) > \max_{z \in \ell} \psi(z) - \varepsilon \geq \psi(y_i) - \varepsilon.$$
But $|y_i - q| < 2^{-i}$, so $|y_i - x| < 2^{-i + 5}$.

We can now take $i \to \infty$ and apply upper-semicontinuity of $Lu$ and continuity of $\psi$ to deduce that for every $x \in V \setminus Z$,
$$Lu(x) \geq \psi(x) - \varepsilon.$$
Since $Lu(x)$ is upper-semicontinuous and $Z$ is null, this inequality holds for every $x \in V$.
Since $\varepsilon$ was arbitrary, $Lu \geq \psi$, which completes the proof.
\end{proof}

\subsection{Removing the geometric assumptions}
We finally remove Assumption \ref{geometric assumptions}.
We first remove the assumption on the curvature tensor by a scaling argument, and then remove the other assumptions by taking a compact exhaustion of $M \setminus \Gamma$.

\begin{lemma}\label{precompact case}
Theorem \ref{eikonal extension} holds if we in addition assume for $V := M \setminus \Gamma$ that $V$ is precompact in $M$, and $\mathcal M^{d - 1}(\partial V) < \infty$.
\end{lemma}
\begin{proof}
Since $V$ is precompact in $M$, $\vol(V) < \infty$ and $|\Riem_M|$ is bounded on a neighborhood of $\overline V$.
Moreover, Theorem \ref{eikonal extension} is invariant under rescaling the metric on $M$, provided that we rescale $\psi$ as well.
So it is no loss to assume that $|\Riem_M| \leq \varepsilon_0$ on a neighborhood of $\overline V$.
Then we have met Assumption \ref{geometric assumptions}, so the conclusion follows from Lemma \ref{smooth case}.
\end{proof}

\begin{proof}[Proof of Theorem \ref{eikonal extension}]
Let $V := M \setminus \Gamma$, and choose a sequence $(V_n)$ of precompact open subsets of $V$ such that:
\begin{enumerate}
    \item $V_{n + 1} \supseteq V_n$.
    \item $\bigcup_n V_n = V$.
    \item $V_0 = \emptyset$.
    \item $\partial V_n$ is smooth.
\end{enumerate}
Also let $\Gamma_n := \overline V_{n - 1} \cup (M \setminus V_n)$ if $n \geq 1$.
Let $u_0 := f$, and suppose that we are given $n \geq 1$ and maps $u_0, \dots, u_{n - 1}$ such that:
\begin{enumerate}
    \item For each $m \leq n - 1$, $Lu_m = \psi$ on $V_m$.
    \item For each $1 \leq m \leq n - 1$, $u_m = u_{m - 1}$ on $\Gamma_m$.
\end{enumerate}
Since $\partial (V_n \setminus \Gamma_n)$ is a compact smooth hypersurface, $\mathcal M^{d - 1}(\partial (V_n \setminus \Gamma_n)) < \infty$.
By Lemma \ref{precompact case} applied to with $M$ replaced by $V_n$ and $\Gamma$ replaced by $\Gamma_n$, we can find a map $u_n$ such that $u_n = u_{n - 1}$ on $\Gamma_n$ and $Lu_n = \psi$ on $V_n \setminus \overline V_{n - 1}$.

\begin{claim}
For every $x \in \overline V_{n - 1}$, $Lu_n(x) = \psi(x)$.
\end{claim}
\begin{proof}[Proof of claim]
We first observe that the result holds if $x \in V_{n - 1}$, since the germs of $u_n$ and $u_{n - 1}$ at $x$ agree in that case.
So $Lu_n = \psi$ on $V_n \setminus \partial V_{n - 1}$.

Let $x \in \partial V_{n - 1}$.
By upper-semicontinuity of $Lu_n$, continuity of $\psi$, and the fact that the claim holds away from $\partial V_{n - 1}$, we deduce $Lu_n(x) \geq \psi(x)$.
Conversely, since $\partial V_{n - 1}$ is a smooth hypersurface, it is null, so by Lemma \ref{lower semicontinuity on average}, continuity of $\psi$, and the fact that the claim holds away from $\partial V_{n - 1}$, $Lu_n(x) \leq \psi(x)$.
\end{proof}

Therefore $u_n$ satisfies the inductive assumption.
It is clear that for each $m$, $u_n|_{V_m}$ stabilizes when $n = m$.
So $(u_n)$ converges locally uniformly to a map $u$ with $Lu = Lu_n$ on $V_n$.
Therefore $Lu = \psi$.
\end{proof}

\printbibliography

\end{document}